\documentclass[12pt]{amsart}
\usepackage{graphics,amsmath,amsthm,amssymb,xcolor,booktabs}
\usepackage{setspace}
\usepackage[a4paper,margin=1.9cm,top=2.5cm,bottom=2.5cm,centering,vcentering]{geometry}

\usepackage{mathtools,hyperref}

\newtheorem{theorem}{Theorem}[section]

\newtheorem{remark}[theorem]{Remark}
\newtheorem{conjecture}[theorem]{Conjecture}
\newtheorem{lemma}[theorem]{Lemma}

\newcommand{\opt}{\overline{OPT}}
\newcommand{\bt}{\overline{bt}}
\newcommand{\btt}{\overline{b}}

\title[Further Arithmetic Properties of Overcubic Partition Triples]{Further Arithmetic Properties of Overcubic Partition Triples}

\author[M. P. Saikia]{Manjil P. Saikia}
\address[M. P. Saikia]{Mathematical and Physical Sciences division, School of Arts and Sciences, Ahmedabad University, Ahmedabad 380009, Gujarat India}
\email{manjil@saikia.in}

\author[A. Sarma]{Abhishek Sarma}
\address[A. Sarma]{Department  of Mathematical Sciences, Tezpur University, Napaam,  Tezpur 784028, Assam, India}
\email{abhitezu002@gmail.com}



\begin{document}

\begin{abstract}
    In this short note, we prove several new congruences for the overcubic partition triples function, using both elementary techniques and the theory of modular forms. These extend the recent list of such congruences given by Nayaka, Dharmendra, and Kumar (2024). We also generalize overcubic partition triples to overcubic partition $k$-tuples and prove a few arithmetic properties for these type of partitions.
\end{abstract}
\maketitle
\section{Introduction}\label{sec:intro}

An integer partition of $n$ is a sequence of non-increasing parts $\lambda_1\geq \lambda_2\geq \cdots \geq \lambda_k$ such that $\sum\limits_{i=0}^k\lambda_k=n$. We denote by $p(n)$, the number of partitions of any nonnegative integer $n$. For instance, there are $7$ partitions of $5$, namely
\[
5, 4+1, 3+2, 3+1+1, 2+2+1, 2+1+1+1,~\text{and}~1+1+1+1+1,
\]
and hence $p(5)=7$. Euler proved that the generating function of $p(n)$ is given by
\[
\sum_{n\geq 0}p(n)q^n=\frac{1}{\prod_{i\geq 1}(1-q^i)}:=\frac{1}{f_1},
\]
where for $k\geq1$ we define the short-hand notation $f_k:=\prod_{i\geq 1}(1-q^{ki})$. For a general survey of the theory of partitions, we refer the reader to the books of Andrews \cite{AndrewsBook} and Johnson \cite{Johnson}.

Since the time of Euler, mathematicians have studied numerous special subsets as well as generalizations of the set of partitions; and a rich theory has grown up. One such generalization is that of the \textit{cubic partition} function, first studied by Chan \cite{Chan1, Chan2}. A cubic partition of $n$ is a partition of $n$ in which the even parts can appear in two colors. If the number of such partitions of $n$ are denoted by $a(n)$, then its generating function is given by
\[
\sum_{n\geq 0}a(n)q^n=\frac{1}{f_1f_2}.
\]

Shortly after Chan's work, Kim \cite{Kim} studied the overpartition analogue of the cubic partition, called the \textit{overcubic partition} function. The number of overcubic partitions of $n$, denoted by $\overline{a}(n)$ counts the number of overlined version of the cubic partitions counted by $a(n)$: that is, the cubic partitions where the first instance of each part is allowed to be overlined. The generating function is given by
\begin{equation}\label{gf-an}
  \sum_{n\geq 0}\overline{a}(n)q^n=\frac{f_4}{f_1^2f_2}.  
\end{equation}

Zhao and Zhong \cite{ZhaoZhong} subsequently studied the number of \textit{cubic partition pairs}, denoted by $b(n)$ with the following generating function
\[
\sum_{n\geq 0}b(n)q^n=\frac{1}{f_1^2f_2^2}.
\]
The overpartitions version of this function was studied by Kim \cite{Kim2}, who denoted by $\overline{b}(n)$ the number of overcubic partition pairs of $n$. The generating function for this function is given by
\[
\sum_{n\geq 0}\overline{b}(n)q^n=\frac{f_4^2}{f_1^4f_2^2}.
\]
As can be seen, we can extend this definition even further, as was done recently by Nayaka, Dharmendra, and Kumar \cite{NayakaDharmendraKumar}. They denoted by $\bt(n)$ the number of overcubic partition triples of $n$, and gave the following generating function
\begin{align}\label{eq-gf}
    \sum_{n\geq 0}\bt(n)q^n=\frac{f_4^3}{f_1^6f_2^3}.
\end{align}

One of the fundamental question that arises in the theory of integer partitions is whether $p(n)$ or any interesting subclass of partitions satisfy any nice arithmetic properties. For $p(n)$ this answer is found in the following celebrated Ramanujan's congruences
\begin{align*}
    p(5n+4)&\equiv 0 \pmod 5,\\
    p(7n+5)&\equiv 0\pmod 7,\\
    p(11n+6)&\equiv 0\pmod{11}.
\end{align*}
Following this cue from Ramanujan, several mathematicians have studied such `Ramanujan-type congruences' for many sub-classes of partitions as well as for generalizations of partitions, including for $a(n), \overline{a}(n), b(n), \overline{b}(n)$, and very recently for $\bt(n)$ by Nayaka, Dharmendra, and Kumar \cite{NayakaDharmendraKumar}. The main goal of this paper is to extend the list of congruences given by  Nayaka, Dharmendra, and Kumar \cite{NayakaDharmendraKumar}. We also extend the definition of overcubic partition triples to overcubic partition $k$-tuples and explore arithmetic properties of this class of partitions.

Before stating our main results, we note that Nayaka, Dharmendra, and Kumar \cite[Eq. (46)]{NayakaDharmendraKumar} showed that for all $n\geq 1$, we have
\[
\bt(2n+1)\equiv 0 \pmod 2.
\]
In fact, we have for all $n\geq 1$
\begin{equation}\label{pmod2}
\bt(n)\equiv 0\pmod 2.
\end{equation}
This follows easily from the binomial theorem by observing that
\[
\sum_{n\geq 0}\bt(n)q^n=\prod_{i\geq 1}\left(\frac{1+q^{2i}}{1+q^{2i}-2q^i}\right)^3=\prod_{i\geq 1}\left(1+2\frac{q^i}{1+q^{2i}-2q^i}\right)^3\equiv 1\pmod 2.
\]
We also note the following congruence, which we will use without commentary in the sequel: for a prime $p$, and positive integers $k$ and $l$, we have
     \begin{align*}
         f_{k}^{p^l} \equiv f_{pk}^{p^{l-1}} \pmod{p^l}.
     \end{align*}

We now state our first result.
\begin{theorem}\label{thm1}
    For all $n\geq 0$, we have
    \begin{align}
    \bt(4n+3)&\equiv 0 \pmod{4},\label{cong-0}\\
        \bt(8n+5)&\equiv 0 \pmod{32},\label{cong-1}\\
        \bt(8n+6)&\equiv 0 \pmod{4},\label{cong-01}\\
        \bt(8n+7)&\equiv 0 \pmod{64}\label{cong-2},\\
    \bt(16n+10)&\equiv 0 \pmod{32},\label{cong-3}\\
    \bt(16n+12)&\equiv 0 \pmod{4},\label{cong-4}\\
\bt(16n+14)&\equiv 0 \pmod{64}\label{cong-5},\\
\bt(32n+20)&\equiv 0 \pmod{32}\label{cong-6},\\
\bt(32n+24)&\equiv 0 \pmod{4}\label{cong-7},\\
\bt(32n+28)&\equiv 0 \pmod{64}\label{cong-8}.
    \end{align}
\end{theorem}
\begin{remark}
Some of our congruences are better than those found by     Nayaka, Dharmendra, and Kumar. For instance, they had proved \cite[Theorem 1]{NayakaDharmendraKumar} for all $n\geq 0$
    \[
    \bt(8n+5)\equiv 0\pmod 8 \quad \text{and} \quad \bt(8n+7)\equiv 0\pmod{32}.
    \]
    They also proved \cite[Theorem 4]{NayakaDharmendraKumar}, for all $n\geq 0$
   \[
    \bt(16n+10)\equiv 0\pmod{16} \quad \text{and} \quad \bt(16n+14)\equiv 0\pmod{16}.
    \]
    Finally, they had also proved \cite[Theorem 5]{NayakaDharmendraKumar}
   \[
    \bt(32n+20)\equiv 0\pmod{16} \quad \text{and} \quad \bt(32n+28)\equiv 0\pmod{16}.
    \]
\end{remark}

There are numerous other congruences modulo powers of $2$ and multiples of $3$ that the $\bt(n)$ function satisfies. We prove the following two such congruences here to give a flavour.
\begin{theorem}\label{thm3}
    For all $n\geq 0$, we have
      \begin{align}
        \bt(72n+21)&\equiv 0 \pmod{128},\label{cong-1-thm3}\\
     \bt(72n+69)&\equiv 0 \pmod{384}\label{cong-2-thm3}.
    \end{align}  
\end{theorem}
\noindent We prove Theorems \ref{thm1} and \ref{thm3} in Section \ref{sec:proof-thm:cong}, using Smoot's implementation \cite{Smoot} of Radu's algorithm \cite{Radu2,Radu} coming from the theory of modular forms.

The main purpose of proving Theorem \ref{thm1} was to `guess' the following theorem.
\begin{theorem}\label{thm2}
    For all $n\geq 0$ and $\alpha\geq 0$, we have
    \begin{align}
        \bt(2^\alpha(4n+3))&\equiv 0 \pmod 4,\label{cong-1-thm2}\\
        \bt(2^\alpha(8n+5))&\equiv 0 \pmod{32}\label{cong-2-thm2}.
    \end{align}
\end{theorem}
\noindent We prove Theorem \ref{thm2} in Section \ref{sec:proof-thm2} using elementary means. We can go even further, but we leave the following as an open problem.
\begin{conjecture}
For all $n\geq 0$ and $\alpha\geq 0$, we have
\[   \bt(2^\alpha(8n+7))\equiv 0 \pmod{64}.\label{cong-3-thm2}\]
\end{conjecture}

The main purpose of proving Theorem \ref{thm3} was to `guess' the following conjecture.
\begin{conjecture}\label{conj2}
    For all $n\geq 0$ and $\alpha\geq 0$, we have
    \begin{align}
    \bt(144n+42)&\equiv 0 \pmod{384},\label{cong-3-thm3}\\
        \bt(2^\alpha(72n+21))&\equiv 0 \pmod{128},\label{cong-1-thm4}\\
        \bt(2^\alpha(72n+69))&\equiv 0 \pmod{128}.\label{cong-2-thm4}
    \end{align}
\end{conjecture}

As the reader has probably guessed, there is nothing special in taking overcubic partition pairs or triples. We can define $\btt_k(n)$ to be the number of overcubic partition $k$-tuples, with the following generating function
\begin{equation}\label{eq-gf-2}
    \sum_{n\geq 0}\btt_k(n)q^n=\frac{f_4^k}{f_1^{2k}f_2^k}.
\end{equation}
Then $\btt_1(n)=\overline{a}(n), \btt_2(n)=\btt(n)$ and $\btt_3(n)=\bt(n)$. Similar to the $\bt(n)$ function, there seems to be many congruences that the $\btt_k(n)$ function satisfies for powers of $2$. We just state a few of them here, in the results below.

We begin with the following easy to prove result.
\begin{theorem}\label{thm:9}
    For all $n\geq 0$ and $k\geq 1$ with $n, k \in \mathbb{Z}$ we have
    \[
    \btt_{2k+1}(n)\equiv \overline{a}(n)\pmod4.
    \]
\end{theorem}
\begin{proof}
    We have
 \begin{equation}\label{cong2k+1}
        \sum_{n\geq 0}\btt_{2k+1}(n)q^n
        =\frac{f_{4}^{2k+1}}{f_{1}^{2(2k+1)}f_{2}^{2k+1}} = \frac{f_{4}^{2k}f_{4}}{f_{1}^{4k}f_{2}^{2k}f_{1}^2f_{2}} \equiv \frac{f_{4}}{f_{1}^{2}f_{2}} \pmod4.
    \end{equation}
This completes the proof, via \eqref{gf-an}.
\end{proof}
\noindent Theorem \ref{thm:9} can be used to prove an infinite family of Ramanujan-like congruences modulo $4$ for the functions $\btt_{2k+1}$. We do not explore this aspect in this paper.

Our next theorem gives a general modulo $4$ congruence for the overcubic partition $k$-tuples function.
\begin{theorem} 
\label{Ram_congs_mod4}
For all $n\geq 0$, $k\geq 0$ with $n, k \in \mathbb{Z}$ and $p\geq 3$ prime, and all quadratic nonresidues $r$ modulo $p$ with $1\leq r\leq p-1$ we have
    $$
\btt_{2k+1}(2pn+R) \equiv 0 \pmod 4 ,
    $$
    where 
    $$
    R = \begin{cases}
			r, & \text{if $r$ is odd,}\\
            p+r, & \text{if $r$ is even.}
		 \end{cases}
   $$
\end{theorem}
\noindent Again the proof is not difficult, so we complete it here.
\begin{proof}[Proof of Theorem \ref{Ram_congs_mod4}]
From \cite[Theorem 2.5]{overcubicsellers}, we know that, for all $n\geq 1,$ $\overline{a}(n) \equiv 0 \pmod{4}$ if and only if $n$ is neither a square nor twice a square. So it is enough for us to show that $2pn+R$ as defined above is never a square and never twice a square, thanks to Theorem \ref{thm:9}. Clearly, $2pn+R$ is always odd by definition, so it cannot be twice a square. Next, from the definition of $R$ we see that $2pn+R \equiv r \pmod{p}$. Since $r$ is defined to be a quadratic nonresidue modulo $p$, we know that $r$ cannot be congruent to a square modulo $p.$  Thus, $2pn+R$ cannot equal a square.  This concludes the proof.
\end{proof}

Our final congruence result is the following result.
\begin{theorem}\label{thm5}
    For all $n\geq 0$ and $k\geq 0$, we have
    \begin{align}
        \btt_{2k+1}(8n+1)&\equiv 0\pmod2,\\
    \btt_{2k+1}(8n+2)&\equiv 0\pmod2,\\
    \btt_{2k+1}(8n+3)&\equiv 0\pmod4,\\
    \btt_{2k+1}(8n+4)&\equiv 0\pmod2,\\
    \btt_{2k+1}(8n+5)&\equiv 0\pmod{8},\\
    \btt_{2k+1}(8n+6)&\equiv 0\pmod4,\\
    \btt_{2k+1}(8n+7)&\equiv 0\pmod{16}.
    \end{align}
\end{theorem}
\noindent We give a sketch proof of this theorem in the paper in Section \ref{sec:sketchpf}. In fact, it is easy to verify the modulo $2$ congruences, like we did with \eqref{pmod2}.

Along with the study of Ramanujan-type congruences, the study of distribution of the coefficients of a formal power series modulo $M$ is also an interesting problem to explore. Given an integral power series $A(q) := \displaystyle\sum_{n=0}^{\infty}a(n)q^n $ and $0 \leq r \leq M$, the arithmetic density $\delta_r(A,M;X)$ is defined as
\begin{equation*}
       \delta_r(A,M;X) = \frac{\#\{ n \leq X : a(n) \equiv r \pmod{M} \}}{X}.
  \end{equation*}
An integral power series $A$ is called \textit{lacunary modulo $M$} if
\begin{equation*}
   \lim_{X \to \infty} \delta_0(A,M;X)=1,
\end{equation*}
which means that almost all the coefficients of $A$ are divisible by $M$. It turns out that the partition function $\bt(n)$ is almost always divisible by $2^k$ for $k\geq1$. Specifically, we prove the following result in Section \ref{sec:lacunary}.
\begin{theorem}\label{thm:lacunary}
	For $\ell\geq1$, let $G(q)=\sum_{n=0}^{\infty}\overline{b}_{\ell}(n)q^n$. Then for every positive integer $k$, 
 \begin{equation}\label{lacunary}
     \lim_{X\to\infty} \delta_{0}(G,2^k;X)  = 1.
 \end{equation}
	\end{theorem}

This paper is organized as follows: Sections \ref{sec:proof-thm:cong} -- \ref{sec:lacunary} contains the proofs of our results, and we end our paper with some concluding remarks in Section \ref{sec:conc}.

\section{Proofs of Theorems \ref{thm1} and \ref{thm3}}\label{sec:proof-thm:cong}

In this section we prove Theorems \ref{thm1} and \ref{thm3} using an algorithmic approach. More specifically, we use Smoot's \cite{Smoot} implementation of Radu's algorithm \cite{Radu2,Radu}, which can be used to prove Ramanujan type congruences of the form stated in Section \ref{sec:intro}. The algorithm takes as an input the generating function
\[
\sum_{n= 0}^\infty a_r(n)q^n=\prod_{\delta|M}\prod_{n= 1}^\infty (1-q^{\delta n})^{r_\delta},
\]
and positive integers $m$ and $N$, with $M$ another positive integer and $(r_\delta)_{\delta|M}$ is a sequence indexed by the positive divisors $\delta$ of $M$. With this input, Radu's algorithm tries to produce a set $P_{m,j}(j)\subseteq \{0,1,\ldots, m-1\}$ which contains $j$ and is uniquely defined by $m, (r_\delta)_{\delta|M}$ and $j$. Then, it decides if there exists a sequence $(s_\delta)_{\delta |N}$ such that
\[
q^\alpha \prod_{\delta|M}\prod_{n=1}^\infty (1-q^{\delta n})^{s_\delta} \cdot \prod_{j^\prime \in P_{m,j}(j)}\sum_{n=0}^\infty a(mn+j^\prime)q^n,
\]
is a modular function with certain restrictions on its behaviour on the boundary of $\mathbb{H}$.

Smoot \cite{Smoot} implemented this algorithm in Mathematica and we use his \texttt{RaduRK} package, which requires the software package \texttt{4ti2}. Documentation on how to intall and use these packages are available from Smoot \cite{Smoot}. We use this implemented \texttt{RaduRK} algorithm to prove Theorem \ref{thm1} in the next section.

It is natural to guess that $N=m$ (which corresponds to the congruence subgroup $\Gamma_0(N)$), but this is not always the case, although they are usually closely related to one another. The determination of the correct value of $N$ is an important problem for the usage of \texttt{RaduRK} and it depends on the $\Delta^\ast$ criterion described in the previous subsection. It is easy to check the minimum $N$ which satisfies this criterion by running \texttt{minN[M, r, m, j]}. The generating function of $\bt(n)$ given in \eqref{eq-gf} can be described by setting $M=4$ and $r=\{-6,-3,3\}$.

\begin{proof}[Proof of Theorem \ref{thm1}]
 Since the proof of all congruences listed in Theorem \ref{thm1} are similar, we only prove \eqref{cong-2} here. The rest of the output can be obtained by visiting \url{https://manjilsaikia.in/publ/mathematica/Overcubic_Triples.nb}. We use \eqref{eq-gf} and calculate \texttt{minN[4,\{-6,-3,3\},8,7]}, which gives $N=8$, which is easily handled in a modest laptop. Radu's algorithm now gives a straight proof of \eqref{cong-2}. Here we give the output of \texttt{RK}.

\allowdisplaybreaks{
	\begin{doublespace}
		\begin{align*}
		\texttt{In[1] := } & \texttt{RaduRK[8,4,\{-6,-3,3\},8,7]}\\
		& \prod_{\texttt{$\delta$|M}} (\texttt{q}^{\delta };\texttt{q}^{\delta })_{\infty }^{\texttt{r}_{\delta }}  = \sum_{\texttt{n=0}}^{\infty } \texttt{a}(\texttt{n})\,\texttt{q}^\texttt{n}\\
		& \fbox{$\texttt{f}_\texttt{1}(\texttt{q})\cdot \prod\limits_{\texttt{j}'\in \texttt{P}_{\texttt{m,r}}(\texttt{j}) } \sum\limits_{\texttt{n=0}}^\infty \texttt{a}(\texttt{mn}+\texttt{j}')\,\texttt{q}^\texttt{n} = \sum\limits_{\texttt{g}\in \texttt{AB}} \texttt{g}\cdot \texttt{p}_\texttt{g}(\texttt{t}) $} \\
		& \texttt{Modular Curve: }\texttt{X}_\texttt{0}(\texttt{N}) \\
		\texttt{Out[2] = }\\
  &\begin{array}{c|c}
 \text{N:} & 8 \\
\hline
 \text{$\{$M,(}r_{\delta })_{\delta |M}\text{$\}$:} & \{4,\{-6,-3,3\}\} \\
\hline
 \text{m:} & 8 \\
\hline
 P_{m,r}\text{(j):} & \{7\} \\
\hline
 f_1\text{(q):} & \dfrac{(q;q)_{\infty }^{69} \left(q^4;q^4\right)_{\infty }^{30}}{q^{15}
   \left(q^2;q^2\right)_{\infty }^{29} \left(q^8;q^8\right)_{\infty }^{64}} \\
\hline
 \text{t:} & \dfrac{\left(q^4;q^4\right)_{\infty }^{12}}{q \left(q^2;q^2\right)_{\infty
   }^4 \left(q^8;q^8\right)_{\infty }^8} \\
\hline
 \text{AB:} & \{1\} \\
\hline
 \left\{p_g\text{(t): g$\in $AB$\}$}\right. & \left\{9792 t^{15}+6606336 t^{14}+905825280
   t^{13}\right.\\ &+46058225664 t^{12}+1124900028416 t^{11}\\&+15177685794816 t^{10}+122507156520960
   t^9\\&+616578030764032 t^8+1960114504335360 t^7\\&+3885487563472896 t^6+4607590516391936
   t^5\\&+3018471877115904 t^4+949826648801280 t^3\\&\left.+110835926040576 t^2+2628519985152
   t\right\} \\
\hline
 \text{Common Factor:} & 64 \\
\end{array}
		\end{align*}
	\end{doublespace}  }

The interpretation of this output is as follows.

The first entry in the procedure call \texttt{RK[8, 4, \{-6, -3, 3\}, 8, 7]} corresponds to specifying $N=8$, which fixes the space of modular functions
\[
M(\Gamma_0(N)):=\text{the algebra of modular functions for $\Gamma_0(N)$}.
\]

 The second and third entry of the procedure call \texttt{RK[8, 4, \{-6, -3, 3\}, 8, 7]} gives the assignment $\{M, (r_\delta)_{\delta|M}\}=\{4, (-6, -3, 3)\}$, which corresponds to specifying $(r_\delta)_{\delta|M}=(r_1,r_2,r_4)=(-6,-3,3)$, so that
 \[
\sum_{n\geq 0}\bt(n)q^n=\prod_{\delta|M}(q^\delta;q^\delta)^{r_\delta}_\infty = \frac{f_4^3}{f_1^6f_2^3}.
 \]

 The last two entries of the procedure call \texttt{RK[8, 4, \{-6, -3, 3\}, 8, 7]} corresponds to the assignment $m=8$ and $j=7$, which means that we want the generating function
 \[
\sum_{n\geq 0}\bt(n)(mn+j)q^n=\sum_{n\geq 0}\bt(n)(8n+7)q^n.
 \]
 So, $P_{m,r}(j)=P_{8,r}(7)$ with $r=(-6,-3,3)$.

The output $P_{m,r}(j):=P_{8,(-6,3,3)}(7)=\{7\}$ means that there exists an infinite product
\[
f_1(q)=\dfrac{(q;q)_{\infty }^{69} \left(q^4;q^4\right)_{\infty }^{30}}{q^{15}
   \left(q^2;q^2\right)_{\infty }^{29} \left(q^8;q^8\right)_{\infty }^{64}},
\]
such that
\[
f_1(q)\sum_{n\geq 0}\bt(n)(8n+7)q^n\in M(\Gamma_0(8)).
\]

Finally, the output
\[
t=\dfrac{\left(q^4;q^4\right)_{\infty }^{12}}{q \left(q^2;q^2\right)_{\infty
   }^4 \left(q^8;q^8\right)_{\infty }^8}, \quad AB=\{1\}, \quad \text{and}\quad \{p_g\text{(t): g$\in AB$\}},
\]
presents a solution to the question of finding a modular function $t\in M(\Gamma_0(8))$ and polynomials $p_g(t)$ such that
\[
f_1(q)\sum_{n\geq 0}\bt(8n+7)q^n =\sum_{g\in AB}p_g(t)\cdot g
\]
In this specific case, we see that the singleton entry in the set $\{p_g\text{(t): g$\in AB$\}}$ has the common factor $64$, thus proving equation \eqref{cong-2}.
\end{proof}

\begin{remark}
    The interested reader can refer to \cite{Saikia} or \cite{AndrewsPaule2} for some more recent applications of the method.
\end{remark}

\begin{proof}[Proof of Theorem \ref{thm3}]
    Since the proof of Theorem \ref{thm3} is similar to the proof of Theorem \ref{thm1}, we just refer the reader to the output file available at \url{https://manjilsaikia.in/publ/mathematica/Overcubic_Triples_72.nb}
\end{proof}

 \section{Proof of Theorem \ref{thm2}}\label{sec:proof-thm2}

Some known 2-dissections (see for example \cite[Lemma 2]{matching}) are stated in the following lemma, which will be used subsequently.
\begin{lemma}\label{lem-diss}
We have
\begin{align}
f_{1}^2 &= \frac{f_{2}f_{8}^5}{f_{4}^2f_{16}^2} - 2 q \frac{f_{2} f_{16}^2}{f_{8}},\label{disf1^2}\\
\frac{1}{f_{1}^2} &= \frac{f_{8}^5}{f_{2}^5 f_{16}^2} + 2 q \frac{f_{4}^2 f_{16}^2 }{f_{2}^5 f_{8}},\label{dis1byf1^2}\\
f_{1}^4 &= \frac{f_{4}^{10}}{f_{2}^2f_{8}^4} - 4 q \frac{f_{2}^2 f_{8}^4}{f_{4}^2},\label{disf1^4}\\
\frac{1}{f_{1}^4} &= \frac{f_{4}^{14}}{f_{2}^{14} f_{8}^4} + 4 q \frac{f_{4}^2 f_{8}^4}{f_{2}^{10}}.\label{dis1byf1^4}
\end{align}
\end{lemma}

\begin{proof}[Proof of \eqref{cong-1-thm2}]
 Nayaka, Dharmendra, and Kumar had found \cite[Eq. (45)]{NayakaDharmendraKumar}
 \begin{equation}\label{eq:1}
     \sum_{n\geq 0}\bt(n)q^n=\frac{f_4^3f_8^{15}}{f_2^{18}f_{16}^6}+6q\frac{f_4^5f_8^9}{f_2^{18}f_{16}^2}+12q^2\frac{f_4^7f_8^3f_{16}^2}{f_2^{18}}+8q^3\frac{f_4^9f_{16}^6}{f_2^{18}f_8^3}.
 \end{equation}
From \eqref{eq:1} we have
\begin{equation}\label{eq:2}
    \sum_{n\geq 0}\bt(2n)q^n\equiv \frac{f_2^3f_4^{15}}{f_1^{18}f_8^6} \pmod 4.
\end{equation}
Using \eqref{eq-gf} and \eqref{eq:2} we obtain, for all $n\geq 0$
\begin{equation}\label{eq3}
    \bt(2n)\equiv \bt(n) \pmod 4.
\end{equation}
Combining \eqref{eq3} with \eqref{cong-0} we obtain \eqref{cong-1-thm2}.
\end{proof}

\begin{remark}
    We note here, it can be proved that $\bt(n)\equiv\overline{a}(n)\pmod{4}$. With this and a result of Sellers \cite[Corollary 2.6]{overcubicsellers} for $\overline{a}(n)$, we get an alternate proof of \eqref{cong-1-thm2}.
\end{remark}

\begin{proof}[Proof of \eqref{cong-2-thm2}]
We re-write \eqref{eq-gf} as
\begin{align*}
    \sum_{n\geq 0}\bt(n)q^n=\frac{f_{4}^3}{f_2^3}\cdot\frac{1}{(f_1^4)^2}\cdot f_1^2.
\end{align*}
Employing \eqref{disf1^2}, \eqref{dis1byf1^4} and then extracting the even powered terms of $q$, we have
\begin{align*}
    \sum_{n\geq 0}\bt(2n)q^n&=\frac{f_2^{29}}{f_1^{30} f_4^3 f_8^2}+16q\frac{f_4^{13} f_2^5 }{f_1^{22} f_8^2}-16q\frac{f_8^2 f_2^{19}}{f_1^{26} f_4}\\
    &\equiv \frac{f_1^2f_2^{13}}{f_4^3 f_8^2}+16qf_1^{24}-16qf_1^{24} \equiv \frac{f_1^2f_2^{13}}{f_4^3 f_8^2}\pmod{32}.
\end{align*}
Extracting the even powered terms of $q$ by using \eqref{disf1^2}, we have
\begin{equation}
    \sum_{n\geq 0}\bt(4n)q^n\equiv \frac{f_1^{14} f_4^3}{f_2^5 f_8^2} \pmod{32}.\label{bt40mod32}
\end{equation}

Again, re-writing \eqref{eq-gf} as
\begin{align*}
    \sum_{n\geq 0}\bt(n)q^n=\frac{f_{4}^3}{f_2^3}\cdot\frac{1}{f_1^4}\cdot\frac{1}{ f_1^2}.
\end{align*}
Employing \eqref{dis1byf1^2}, \eqref{dis1byf1^4} and then extracting the even powered terms of $q$, we have
\begin{align*}
     \sum_{n\geq 0}\bt(2n)q^n&=\frac{f_4 f_2^{17}}{f_1^{22} f_8^2}+8q\frac{f_4^3 f_8^2 f_2^7}{f_1^{18}} \equiv \frac{f_2 f_4 f_1^{10}}{f_8^2} +8q f_2 f_4 f_1^{18} \pmod{32}\\
     &\equiv \frac{f_2 f_4}{f_8^2}\cdot (f_1^4)^2\cdot f_1^2 + 8 q f_2 f_4\cdot (f_1^{4})^4\cdot f_1^2 \pmod{32}.
\end{align*}
Employing \eqref{disf1^2}, \eqref{disf1^4} and further extracting the even powered terms of $q$, we have
\begin{align*}
     \sum_{n\geq 0}\bt(4n)q^n&\equiv\frac{f_2^{19}}{f_1^2 f_4^5 f_8^2}+16q\frac{f_1^2 f_8^2 f_2^9}{f_4^3}+16q\frac{f_1^6 f_4^{11}}{f_2^5 f_8^2}-16q\frac{f_8^2 f_2^{41}}{f_1^6 f_4^{17}}\\
     &\equiv \frac{f_2^{19}}{f_1^2 f_4^5 f_8^2}+16qf_1^{24}+16qf_1^{24}-16qf_1^{24}\\
     &\equiv \frac{f_2^{19}}{f_1^2 f_4^5 f_8^2}+16qf_{16}f_8 \pmod{32}.
\end{align*}
Finally, employing \eqref{dis1byf1^2} and extracting the terms involving $q^{2n}$, we have
\begin{align}
    \sum_{n\geq 0}\bt(8n)q^n&\equiv\frac{f_1^{14} f_4^3}{f_2^5 f_8^2}\pmod{32}.\label{bt8nmod32}
\end{align}

From \eqref{bt40mod32} and \eqref{bt8nmod32}, we have
\begin{align}
    \bt(4n)\equiv\bt(8n)\pmod{32}.\label{cong4n8n:mod32}
\end{align}

Combining \eqref{cong-1}, \eqref{cong-3}, \eqref{cong-6} and \eqref{cong4n8n:mod32}, we conclude the proof.
\end{proof}

\section{Proof of Theorem \ref{thm:lacunary}}\label{sec:lacunary}

Before going into the proof of Theorem \ref{thm:lacunary}, we recall some fundamental ideas from the theory of modular forms. We recall that the Dedekind's eta-function $\eta(z)$ is defined by
\begin{align*}
	\eta(z):=q^{1/24}(q;q)_{\infty}=q^{1/24}\prod_{n=1}^{\infty}(1-q^n),
\end{align*}
where $q:=e^{2\pi iz}$ and $z\in \mathbb{H}$. A function $f(z)$ is called an $eta$-quotient if it can be expressed as a finite product of
the form $$f(z)=\prod_{\delta\mid N}\eta(\delta z)^{r_\delta},$$ where $N$ is a positive integer and each $r_{\delta}$ is an integer. 

Define
\begin{align}
    G(\tau) :=\frac{\eta(\delta_1\tau)^{r_1}\eta(\gamma_1\tau)^{r_2}\cdots\eta(\gamma_u\tau)^{r_u}}{\eta(\gamma_1\tau)^{s_1}\eta(\gamma_2\tau)^{s_2}\cdots\eta(\gamma_t\tau)^{s_t}},\label{G(tau)}
\end{align}
where $r_i , s_i, \delta_i,$ and $\gamma_i$ are positive integers with $\delta_1,...,\delta_u, \gamma_1,...,\gamma_t$ distinct, $u, t \geq 0$. The weight of $G(\tau)$ is given by
\begin{align*}
    \frac{1}{2}\left(\sum_{i=1}^{u}r_i-\sum_{i=1}^{t}s_i\right).
\end{align*}
Also define $\mathcal{D_G}:=\gcd(\delta_1, \delta_2,\ldots,\delta_u)$.
The following result by Cotron \textit{et. al} \cite{cotronetal} will be useful in proving Theorem \ref{thm:lacunary}.
\begin{theorem}\cite[Theorem 1.1]{cotronetal}\label{thm:cotron}
    Suppose $G(\tau)$ is an eta-quotient of the form \eqref{G(tau)} with integer weight. If $p$ is a prime such that $p^a$ divides $\mathcal{D_G}$ and
    \begin{align}
        p^a\geq\sqrt{\frac{\sum_{i=1}^{t}\gamma_is_i}{\sum_{i=1}^{u}\frac{r_i}{\delta_i}}},
    \end{align}
then $G(\tau)$ is lacunary modulo $p^j$ for any positive integer $j$. Moreover, there exists a positive constant $\alpha$, depending on $p$ and $j$, such that the number of integers $n \leq X$ with $p^j$ not dividing $b(n)$ is $O\left(\dfrac{X}{\log^{\alpha}X}\right)$.
\end{theorem}

\begin{proof}[Proof of Theorem \ref{thm:lacunary}]
From \eqref{eq-gf-2}, we recall
\begin{align*}
    \sum_{n\geq 0}\btt_{\ell}(n)q^n=\frac{f_4^{\ell}}{f_1^{2{\ell}}f_2^{\ell}}=\frac{\eta^{\ell}(4z)}{\eta^{2\ell}(z)\eta^{\ell}(2z)}.
\end{align*}
Following the notations used in \eqref{G(tau)} and the paragraph succeeding it, we have \[\delta_1=4, r_1=\ell, \quad \gamma_1=1,\quad s_1=2{\ell},\quad \gamma_2=2, \quad \text{and} \quad s_2=\ell.\] Also $\mathcal{D_G}=4$ and the weight is $-\ell \in \mathbb{Z}$. Next, we see that $2^2|4$ and
\begin{align*}
    2^2\geq\sqrt{\frac{2\ell+2\ell}{\frac{\ell}{4}}}=\sqrt{16}=4.
\end{align*}
Choosing $p=2$ and $a=2$ in Theorem \ref{thm:cotron}, we complete the proof.
\end{proof}

\section{Sketch Proof of Theorem \ref{thm5}}\label{sec:sketchpf}
From the work of Sellers \cite[Corollary 2.3]{overcubicsellers} we can write
\[
\sum_{n\geq 0}\btt_{t}(n)q^n=\left(\varphi(q)\prod_{i\geq 1}\varphi(q^{2^i})^{3\cdot 2^{i-1}}\right)^{t},
\]
where \[\varphi(q):=1+2\sum_{n\geq 0}q^{n^2}=\varphi(q^4) + 2q \psi (q^8),\] with $\psi(q):=\sum_{n\geq 0}q^{n(n+1)/2}$. Since $\left(\prod_{i\geq3}\varphi(q^{2^i})\right)^{3\cdot 2^{i-1}\cdot t}$ is a function of $q^8$, it is enough to do the $8$-dissection of the first three terms. Also, as the highest modulus involved in the theorem is $16$ and the other moduli are divisors of $16$, we will prove our result modulo if we consider only modulo $16$ this $8$-dissection. Re-writing
\[
\sum_{n\geq 0}\btt_t(n)q^n=\left(\sum_{j=0}^{7}a_{t,j}q^{j}F_{t,j}(q^8)\right) \left(\prod_{i\geq3}\varphi(q^{2^i})\right)^{3\cdot 2^{i-1}\cdot t},
\]
where $F_{t,j}(q^8)$ is a function of $q^8$ whose power series representation has integer coefficients. It suffices to just prove the following congruences
\begin{align*}
    a_{t,1} &\equiv 0 \pmod{2},\\
    a_{t,2} &\equiv 0 \pmod{2},\\
    a_{t,3} &\equiv 0 \pmod{4},\\
    a_{t,4} &\equiv 0 \pmod{2},\\
    a_{t,5} &\equiv 0 \pmod{8},\\
    a_{t,6} &\equiv 0 \pmod{4},\\
    a_{t,7} &\equiv 0 \pmod{16}.
\end{align*}
We can do this using induction, and follows exactly the same pattern as the proofs of \cite[Theorem 2.2]{sSellers} and \cite[Lemma 6.1]{SSS}, so we omit the details here. The interested reader can look at the accompanying Mathematica file available at \url{https://manjilsaikia.in/publ/mathematica/overcubictuples.nb}.

\section{Concluding Remarks}\label{sec:conc}

\begin{enumerate}
    \item There is a very related function to the $\btt_k(n)$ function, namely overpartition $k$-tuples with odd parts. We denote by $\opt_k(n)$ the number of overpartition $k$-tuples with odd parts of $n$. The generating function is given by
    \[
    \sum_{\geq 0}\opt_k(n)=\frac{f_2^{3k}}{f_1^{2k}f_4^k}.
    \] In a series of papers \cite{SSS, DSS}, the authors and their collaborators have studied several arithmetic properties that this function satisfies. It seems that, some of the techniques used in these papers translate directly for the $\btt_k(n)$ function. In particular, the authors in \cite{DSS} have explored congruences modulo powers of $2$ and $3$ for the $\opt_{2k+1}(n)$ function (with $k\geq 1$), which we believe may lead to similar results for the $\btt_k(n)$ function. Additionally, in \cite{SSS} the authors have also found inifnite family of congruences modulo small powers of $2$ for $\opt_{2k+1}(n)$ (with $k\geq 1$) and $\opt_4(n)$, we believe similar results also hold for $\btt_k(n)$.
    \item We have just scratched the surface for congruences modulo powers of $2$, a systematic study will unearth several more. For instance, we believe that Theorem \ref{thm5} can be strengthened further by taking appropriate values of $k$; however, the techniques that we have used in this paper seem unsuitable to find a more general result. We again leave this as a future direction of research.
    \item Elementary proof of Theorem \ref{thm1} is not difficult to obtain, we took the algorithmic approach for the sake of brevity. The modulo $4$ congruences are not difficult to see, and in fact do follow from elementary dissection formulas. However, the modulo $32$ and $64$ congruences need some more work. Elementary proof of Theorem \ref{thm3} would be further involved, and would probably give some hints on how to prove Conjecture \ref{conj2}.
\end{enumerate}

\section*{Acknowledgements}

Abhishek Sarma was partially supported by an institutional fellowship for doctoral research from Tezpur University, Assam, India. He thanks the funding institution. The authors thank the anonymous referee for taking the time to read through the manuscript and for providing helpful comments.

\end{document}